\documentclass[11pt]{amsart}
\usepackage{amssymb,latexsym}
\input epsf
\usepackage{graphicx}

\theoremstyle{plain}
\newtheorem{theorem}{Theorem}[section]
\newtheorem{proposition}[theorem]{Proposition}
\newtheorem{lemma}[theorem]{Lemma}
\newtheorem{corollary}[theorem]{Corollary}
\newtheorem{conjecture}[theorem]{Conjecture}
\theoremstyle{definition}
\newtheorem{definition}[theorem]{Definition}
\newtheorem{example}[theorem]{Example}

\newtheorem{question}[theorem]{Question}
\newtheorem{remark}[theorem]{Remark}

\numberwithin{equation}{section}

\newcommand{\link}{{\rm link}}
\newcommand{\inte}{{\rm int}}

\newcommand{\st}{{\rm st}}
\newcommand{\RK}{{\rm K}}
\newcommand{\RL}{{\rm L}}
\newcommand{\vset}{{\rm vert}}

\newcommand{\kk}{{\mathbf k}}

\renewcommand{\to}{\rightarrow}

\newcommand{\sm}{{\smallsetminus}}

\begin{document}
\title[Flag subdivisions and $\gamma$-vectors]
{Flag subdivisions and $\gamma$-vectors}

\author{Christos~A.~Athanasiadis}
\address{Department of Mathematics
(Division of Algebra-Geometry)\\
University of Athens\\
Panepistimioupolis\\
15784 Athens, Hellas (Greece)}
\email{caath@math.uoa.gr}

\date{January 9, 2012; revised, May 28, 2012}

\begin{abstract}
The $\gamma$-vector is an important enumerative invariant of a flag
simplicial homology sphere. It has been conjectured by Gal that this
vector is nonnegative for every such sphere $\Delta$ and by Reiner,
Postnikov and Williams that it increases when $\Delta$ is replaced
by any flag simplicial homology sphere which geometrically subdivides
$\Delta$. Using the nonnegativity of the $\gamma$-vector in dimension
3, proved by Davis and Okun, as well as Stanley's theory of simplicial
subdivisions and local $h$-vectors, the latter conjecture is confirmed
in this paper in dimensions 3 and 4.
\end{abstract}

\maketitle

 \section{Introduction}
  \label{sec:intro}

  This paper is concerned with the face enumeration of an important class of
  simplicial complexes, that of flag homology spheres, and their subdivisions.
  The face vector of a homology sphere (more generally, of an Eulerian simplicial
  complex) $\Delta$ can be conveniently encoded by its $\gamma$-vector
  \cite{Ga05}, denoted $\gamma (\Delta)$. Part of our motivation comes from the
  following two conjectures (we refer to Section \ref{sec:sub} for all relevant
  definitions). The first, proposed by Gal \cite[Conjecture 2.1.7]{Ga05}, can be
  thought of as a Generalized Lower Bound Conjecture for flag homology spheres
  (it strengthens an earlier conjecture by Charney and Davis \cite{CD95}). The
  second, proposed by Postnikov, Reiner and Williams \cite[Conjecture
  14.2]{PRW08}, is a natural extension of the first.

    \begin{conjecture} [Gal \cite{Ga05}]  \label{conj:gal}
      For every flag homology sphere $\Delta$ we have $\gamma (\Delta) \ge 0$.
    \end{conjecture}

    \begin{conjecture} [Postnikov, Reiner and Williams \cite{PRW08}]
    \label{conj:monotone}
      For all flag homology spheres $\Delta$ and $\Delta'$ for which $\Delta'$
      geometrically subdivides $\Delta$, we have $\gamma (\Delta') \ge \gamma
      (\Delta)$.
    \end{conjecture}

  The previous statements are trivial for spheres of dimension two or less.
  Conjecture \ref{conj:gal} was proved for 3-dimensional spheres by Davis and
  Okun \cite[Theorem 11.2.1]{DOk01} and was deduced from that result for
  4-dimensional spheres in \cite[Corollary 2.2.3]{Ga05}. Conjecture
  \ref{conj:monotone} can be thought of as a conjectural analogue of the fact
  \cite[Theorem 4.10]{Sta92} that the $h$-vector (a certain linear transformation
  of the face vector) of a Cohen-Macaulay simplicial complex increases under
  quasi-geometric simplicial subdivision (a class of topological subdivisions
  which includes all geometric simplicial subdivisions). The main result of
  this paper proves its validity in three and four dimensions for a new class
  of simplicial subdivisions, which includes all geometric ones.

    \begin{theorem} \label{thm:34}
      For every flag homology sphere $\Delta$ of dimension 3 or 4 and for every
      flag vertex-induced homology subdivision $\Delta'$ of $\Delta$, we have
      $\gamma (\Delta') \ge \gamma (\Delta)$.
    \end{theorem}

  The previous result naturally suggests the following stronger version of
  Conjecture \ref{conj:monotone}.

    \begin{conjecture} \label{conj:newmonotone}
      For every flag homology sphere $\Delta$ and every flag vertex-induced
      homology subdivision $\Delta'$ of $\Delta$, we have $\gamma (\Delta')
      \ge \gamma (\Delta)$.
    \end{conjecture}

  The following structural result on flag homology spheres, which may be of 
  independent interest, will also be proved in Section \ref{sec:PL}. It implies, 
  for instance, that Conjecture \ref{conj:newmonotone} is stronger than 
  Conjecture \ref{conj:gal}. Throughout this paper, we will denote by 
  $\Sigma_{d-1}$ the boundary complex of the $d$-dimensional cross-polytope 
  (equivalently, the simplicial join of $d$ copies of the zero-dimensional 
  sphere). 

    \begin{theorem} \label{thm:main}
      Every flag $(d-1)$-dimensional homology sphere is a vertex-induced (hence
      quasi-geometric and flag) homology subdivision of $\Sigma_{d-1}$.
    \end{theorem}

  The proof of Theorem \ref{thm:34} relies on the theory of face enumeration
  for simplicial subdivisions, developed by Stanley \cite{Sta92}. Given a
  simplicial complex $\Delta$ and a simplicial subdivision $\Delta'$ of $\Delta$,
  the $h$-vector of $\Delta'$ can be expressed in terms of local contributions,
  one for each face of $\Delta$, and the combinatorics of $\Delta$ \cite[Theorem
  3.2]{Sta92}. The local contributions are expressed in terms of the key concept
  of a local $h$-vector, introduced and studied in \cite{Sta92}. When $\Delta$ is
  Eulerian, this formula transforms to one involving $\gamma$-vectors (Proposition
  \ref{prop:gammaformula}) and leads to the concept of a local $\gamma$-vector,
  introduced in Section \ref{sec:gamma}. Using the Davis-Okun theorem
  \cite{DOk01}, mentioned earlier, it is shown that the local $\gamma$-vector
  has nonnegative coefficients for every flag vertex-induced homology subdivision
  of the 3-dimensional simplex. Theorem \ref{thm:34} is deduced from these
  results in Section \ref{sec:gamma}.

  The proof of Theorem \ref{thm:main} is motivated by that of \cite[Theorem 
  1.2]{Ath11}, stating that the graph of any flag simplicial pseudomanifold of 
  dimension $d-1$ contains a subdivision of the graph of $\Sigma_{d-1}$. 

  We now briefly describe the content and structure of this paper.
  Sections \ref{sec:sub} and \ref{sec:enu} provide the necessary background on
  simplicial complexes, subdivisions and their face enumeration. The notion
  of a homology subdivision, which is convenient for the results of this
  paper, as well as those of a flag subdivision and vertex-induced (a natural
  strengthening of quasi-geometric) subdivision, are introduced in Section
  \ref{subsec:sub}. Section \ref{sec:monolocal} includes a simple example (see
  Example \ref{ex:counter}) which shows that there exist quasi-geometric
  subdivisions of the simplex with non-unimodal local $h$-vector.

  Section \ref{sec:PL} proves Theorem \ref{thm:main} and another structural 
  result on flag subdivisions (Proposition \ref{prop:balltosphere}), stating 
  that every flag vertex-induced homology subdivision of the $(d-1)$-dimensional 
  simplex naturally occurs as a restriction of a flag vertex-induced homology 
  subdivision of $\Sigma_{d-1}$. These results are used in Section \ref{sec:gamma}. 

  Local $\gamma$-vectors are introduced in Section \ref{sec:gamma}, where
  examples and elementary properties are discussed. It is conjectured there that
  the local $\gamma$-vector has nonnegative coordinates for every flag
  vertex-induced homology subdivision of the simplex (Conjecture \ref{conj:main}).
  This statement can be considered as a local analogue of Conjecture
  \ref{conj:gal}. It is shown to imply both Conjectures \ref{conj:gal} and
  \ref{conj:newmonotone} and to hold in dimension three. Section \ref{sec:gamma}
  concludes with the proof of Theorem \ref{thm:34}.

  Section \ref{sec:evidence} discusses some special cases of Conjecture
  \ref{conj:main}. For instance, the conjecture is shown to hold for iterated
  edge subdivisions (in the sense of \cite[Section 5.3]{CD95}) of the simplex.

  \section{Flag complexes, subdivisions and $\gamma$-vectors}
  \label{sec:sub}

  This section reviews background material on simplicial complexes, in particular
  on their homological properties and subdivisions. For more information on these
  topics, the reader is referred to \cite{StaCCA}. Throughout this paper, $\kk$
  will be a field which we will assume to be fixed. We will denote by $|S|$ the
  cardinality, and by $2^S$ the set of all subsets, of a finite set $S$.

  \subsection{Simplicial complexes} \label{subsec:complexes}
  All simplicial complexes we consider will be abstract and finite. Thus, given a
  finite set $\Omega$, a \emph{simplicial complex} on the ground set $\Omega$ is a
  collection $\Delta$ of subsets of $\Omega$ such that $F \subseteq G \in \Delta$
  implies $F \in \Delta$. The elements of $\Delta$ are called \emph{faces}. The
  dimension of a face $F$ is defined as one less than the cardinality of $F$. The
  dimension of $\Delta$ is the maximum dimension of a face and is denoted by $\dim
  (\Delta)$. Faces of $\Delta$ of dimension zero or one are called \emph{vertices}
  or \emph{edges}, respectively. A \emph{facet} of $\Delta$ is a face which is
  maximal with respect to inclusion. All topological properties of $\Delta$ we
  mention in the sequel will refer to those of the geometric realization $\|\Delta\|$
  of $\Delta$ \cite[Section 9]{Bj95}, uniquely defined up to homeomorphism. For
  example, we say that $\Delta$ is a simplicial (topological) ball (respectively,
  sphere) if $\|\Delta\|$ is homeomorphic to a ball (respectively, sphere).

  The \emph{open star} of a face $F \in \Delta$, denoted $\st_\Delta (F)$, is
  the collection of all faces of $\Delta$ which contain $F$. The \emph{closed star}
  of $F \in \Delta$, denoted $\overline{\st}_\Delta (F)$, is the subcomplex of
  $\Delta$ consisting of all subsets of the elements of $\st_\Delta (F)$. The
  \emph{link} of the face $F \in \Delta$ is the subcomplex of $\Delta$ defined
  as $\link_\Delta (F) = \{ G \sm F: G \in \Delta, \, F \subseteq G\}$. The
  \emph{simplicial join} $\Delta_1 \ast \Delta_2$ of two collections $\Delta_1$
  and $\Delta_2$ of subsets of disjoint ground sets is the collection whose elements
  are the sets of the form $F_1 \cup F_2$, where $F_1 \in \Delta_1$ and $F_2 \in
  \Delta_2$. If $\Delta_1$ and $\Delta_2$ are simplicial complexes, then so is
  $\Delta_1 \ast \Delta_2$. The simplicial join of $\Delta$ with the
  zero-dimensional complex $\{ \varnothing, \{v\} \}$ is denoted by $v \ast
  \Delta$ and called the \emph{cone over $\Delta$} on the (new) vertex $v$.

  A simplicial complex $\Delta$ is called \emph{flag} if every minimal non-face of
  $\Delta$ has two elements. The closed star and the link of any face of
  a flag complex, and the simplicial join of two (or more) flag complexes, are also
  flag complexes. In particular, the simplicial join of $d$ copies of the
  zero-dimensional complex with two vertices is a flag complex (in fact, a flag
  triangulation of the $(d-1)$-dimensional sphere), which will be denoted by
  $\Sigma_{d-1}$. Explicitly, $\Sigma_{d-1}$ can be described as the simplicial
  complex on the $2d$-element ground set $\Omega_d = \{u_1, u_2,\dots,u_d\} \cup
  \{v_1, v_2,\dots,v_d\}$ whose faces are those subsets of $\Omega_d$ which contain
  at most one element from each of the sets $\{u_i, v_i\}$ for $1 \le i \le d$.

  \subsection{Homology balls and spheres} \label{subsec:homPL}
  Let $\Delta$ be a simplicial complex of dimension $d-1$. We call $\Delta$ a
  \emph{homology sphere} (over $\kk$) if for every $F \in \Delta$ (including $F =
  \varnothing$) we have
    $$ \widetilde{H}_i \, (\link_\Delta (F), \kk) \ = \ \begin{cases}
       \kk, & \text{if $i = \dim \, \link_\Delta (F)$ } \\
       0, & \text{otherwise,} \end{cases} $$
  where $\widetilde{H}_* (\Gamma, \kk)$ denotes reduced simplicial homology of
  $\Gamma$ with coefficients in $\kk$. We call $\Delta$ a \emph{homology ball} (over
  $\kk$) if there exists a subcomplex $\partial \Delta$ of $\Delta$, called the
  \emph{boundary} of $\Delta$, so that the following hold:
    \begin{itemize}
      \item[$\bullet$] $\partial \Delta$ is a $(d-2)$-dimensional homology sphere
                       over $\kk$,
      \item[$\bullet$] for every $F \in \Delta$ (including $F = \varnothing$) we
                       have
        $$ \widetilde{H}_i \, (\link_\Delta (F), \kk) \ = \ \begin{cases}
           \kk, & \text{if $F \notin \partial \Delta$ and $i = \dim \,
           \link_\Delta (F)$ } \\
           0, & \text{otherwise.} \end{cases} $$
    \end{itemize}
  The \emph{interior} of $\Delta$ is defined as $\inte(\Delta) = \Delta$, if
  $\Delta$ is a homology sphere, and $\inte (\Delta) = \Delta \sm \partial \Delta$, if
  $\Delta$ is a homology ball. For example, the simplicial complex $\{ \varnothing,
  \{v\} \}$ with a unique vertex $v$ is a zero-dimensional homology ball (over any
  field) with boundary $\{ \varnothing \}$ and interior $\{ \{v\} \}$. If $\Delta$ is
  a homology ball of dimension $d-1$, then $\partial \Delta$ consists exactly of the
  faces of those $(d-2)$-dimensional faces of $\Delta$ which are contained in a unique
  facet of $\Delta$.

  \begin{remark} \label{rem:homjoin}
  It follows from standard facts \cite[(9.12)]{Bj95} on the homology
  of simplicial joins that the simplicial join of a homology sphere
  and a homology ball, or of two homology balls, is a homology ball
  and that the simplicial join of two homology spheres is again a homology
  sphere. Moreover, in each case the interior of the simplicial join
  is equal to the simplicial join of the interiors of the two
  complexes in question. \qed
  \end{remark}

\subsection{Subdivisions} \label{subsec:sub}
We will adopt the following notion of homology subdivision of an abstract
simplicial complex. This notion generalizes that of topological subdivision of
\cite[Section 2]{Sta92}. We should point out that the class of homology
subdivisions of simplicial complexes is contained in the much broader class of
formal subdivisions of Eulerian posets, introduced and studied in \cite[Section
7]{Sta92}.

\begin{definition} \label{def:sub}
Let $\Delta$ be a simplicial complex. A (finite, simplicial) \emph{homology
subdivision} of $\Delta$ (over $\kk$) is a simplicial complex $\Delta'$ together
with a map $\sigma: \Delta' \to \Delta$ such that the following hold for every
$F \in \Delta$: (a) the set $\Delta'_F := \sigma^{-1} (2^F)$ is a subcomplex of
$\Delta'$ which is a homology ball (over $\kk$) of dimension $\dim(F)$; and
(b) $\sigma^{-1} (F)$ consists of the interior faces of $\Delta'_F$.

Such a map $\sigma$ is said to be a \emph{topological} subdivision if the complex
$\Delta'_F$ is homeomorphic to a ball of dimension $\dim(F)$ for every $F \in
\Delta$.
\end{definition}

  Let $\sigma: \Delta' \to \Delta$ be a homology subdivision of $\Delta$. From the
  defining properties, it follows that the map $\sigma$ is surjective and that
  $\dim (\sigma(E)) \ge \dim(E)$ for every $E \in \Delta'$. Given faces $E \in
  \Delta'$ and $F \in \Delta$, the face $\sigma(E)$ of $\Delta$ is called the
  \emph{carrier} of $E$; the subcomplex $\Delta'_F$ is called the \emph{restriction}
  of $\Delta'$ to $F$. The subdivision $\sigma$ is called \emph{quasi-geometric}
  \cite[Definition 4.1 (a)]{Sta92} if there do not exist $E \in \Delta'$ and face
  $F \in \Delta$ of dimension smaller than $\dim(E)$, such that the carrier of
  every vertex of $E$ is contained in $F$. Moreover, $\sigma$ is called
  \emph{geometric} \cite[Definition 4.1 (b)]{Sta92} if there exists a geometric
  realization of $\Delta'$ which geometrically subdivides a geometric realization
  of $\Delta$, in the way prescribed by $\sigma$.

  Clearly, if $\sigma: \Delta' \to \Delta$ is a homology (respectively, topological)
  subdivision, then the restriction of $\sigma$ to $\Delta'_F$ is also a homology
  (respectively, topological) subdivision of the simplex $2^F$ for every $F \in
  \Delta$. Moreover, if $\sigma$ is quasi-geometric (respectively, geometric), then
  so are all its restrictions $\Delta'_F$ for $F \in \Delta$. As part (c) of the
  following example shows, the restriction of $\sigma$ to a face $F \in \Delta$
  need not be a flag complex, even when $\Delta'$ and $\Delta$ are flag complexes
  and $\sigma$ is quasi-geometric.

  \begin{figure}
  \epsfysize = 2.2 in \centerline{\epsffile{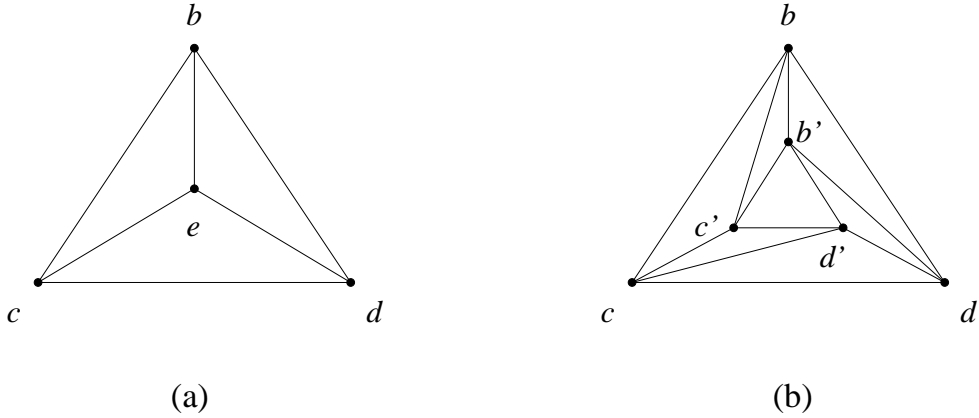}}
  \caption{Two non-flag subdivisions of a triangle.}
  \label{fig:nonregtri}
  \end{figure}

    \begin{example} \label{ex:nonflagrest} Consider a 3-dimensional simplex $2^V$,
    with $V = \{a, b, c, d\}$, and set $F = \{b, c, d\}$.

      (a) Let $\Gamma$ be the simplicial complex consisting of the subsets of $V$
      and the subsets of $\{b, c, d, e\}$ and let $\sigma: \Gamma \to 2^V$ be the 
      subdivision (considered in part (h) of \cite[Example 2.3]{Sta92}) which pushes 
      $\Gamma$ into the simplex $2^V$, so that the face $F$ of $\Gamma$ ends up in 
      the interior of $2^V$ and $e$ ends up in the interior of $2^F$. Formally, for 
      $E \in \Gamma$ we let $\sigma(E) = E$, if $E \in 2^V \sm \{F\}$, we let 
      $\sigma(E) = V$, if $E$ contains $F$ and otherwise we let $\sigma(E) = F$. 
      Then $\Gamma$ is a flag complex and the restriction $\Gamma_F$ of $\sigma$ is 
      the cone over the boundary of $2^F$ (with new vertex $e$), shown in Figure 
      \ref{fig:nonregtri} (a), which is not flag.

      (b) Let $\Gamma'$ be the simplicial complex consisting of the faces of the
      simplex $2^V$ and those of the cone, on a vertex $v$, over the boundary of
      the simplex with vertex set $\{b, c, d, e\}$ (note that $\Gamma'$ is not flag).
      Consider the subdivision $\sigma': \Gamma' \to 2^V$ which satisfies $\sigma'
      (E) = V$ for every face $E \in \Gamma'$ containing $v$ and otherwise agrees
      with the subdivision $\sigma$ of part (a). Then $\sigma'$ is quasi-geometric
      and its restriction $\Gamma'_F = \Gamma_F$ is again the non-flag complex
      shown in Figure \ref{fig:nonregtri} (a).

      (c) Let $\Gamma_0$ be the simplicial complex on the ground set $F \cup \{b',
      c', d'\}$ whose faces are $F$ and those of the simplicial subdivision of $2^F$,
      shown in Figure \ref{fig:nonregtri} (b). Let $\Gamma''$ consist of the faces
      of $2^V$ and those of the cone over $\Gamma_0$ on a new vertex $v$. We leave
      to the reader to verify that $\Gamma''$ is a flag
      simplicial complex and that it admits a quasi-geometric subdivision $\sigma'':
      \Gamma'' \to 2^V$ (satisfying $\sigma''(v) = \sigma''(F) = V$) for which the
      restriction $\Gamma''_F$ is the non-flag simplicial complex shown in Figure
      \ref{fig:nonregtri} (b). \qed
    \end{example}

  The previous examples suggest the following definitions.

    \begin{definition} \label{def:flagsub}
      Let $\Delta', \Delta$ be simplicial complexes and let $\sigma: \Delta' \to
      \Delta$ be a homology subdivision.
        \begin{itemize}
          \item[(i)] We say that $\sigma$ is \emph{vertex-induced} if for all faces
          $E \in \Delta'$ and $F \in \Delta$ the following condition holds: if every
          vertex of $E$ is a vertex of $\Delta'_F$, then $E \in \Delta'_F$.
          \item[(ii)] We say that $\sigma$ is a \emph{flag subdivision} if the
          restriction $\Delta'_F$ is a flag complex for every face $F \in \Delta$.
        \end{itemize}
    \end{definition}

  For homology (respectively, topological) subdivisions we have the hierarchy
  of properties: geometric $\Rightarrow$ vertex-induced $\Rightarrow$ quasi-geometric.
  The subdivision $\Gamma$ of Example \ref{ex:nonflagrest} is not quasi-geometric,
  while $\Gamma'$ and $\Gamma''$ are quasi-geometric but not vertex-induced (none of
  the three subdivisions is flag). Thus the second implication above is strict. An
  example discussed on \cite[p. 468]{Cha94} shows that the first implication is
  strict as well. We also point out here that if $\sigma: \Delta' \to \Delta$ is a
  vertex-induced homology subdivision and the simplicial complex $\Delta'$ is flag,
  then $\sigma$ is a flag subdivision.

\medskip
\noindent
\textbf{Joins and links.} The notion of a (vertex-induced, or flag) homology subdivision
behaves well with respect to simplicial joins and links, as we now explain. Let
$\sigma_1: \Delta'_1 \to \Delta_1$ and $\sigma_2: \Delta'_2 \to \Delta_2$ be homology
subdivisions of two simplicial complexes $\Delta_1$ and $\Delta_2$ on
disjoint ground sets. The simplicial join $\Delta'_1 \ast \Delta'_2$ is naturally a
homology subdivision of $\Delta_1 \ast \Delta_2$ with subdivision map $\sigma:
\Delta'_1 \ast \Delta'_2 \to \Delta_1 \ast \Delta_2$ defined by $\sigma(E_1 \cup E_2)
= \sigma_1 (E_1) \cup \sigma_2 (E_2)$ for $E_1 \in \Delta'_1$ and $E_2 \in \Delta'_2$.
Indeed, given faces $F_1 \in \Delta_1$ and $F_2 \in \Delta_2$, the restriction of
$\Delta'_1 \ast \Delta'_2$ to the face $F = F_1 \cup F_2 \in \Delta_1 \ast \Delta_2$
is equal to $(\Delta'_1)_{F_1} \ast (\Delta'_2)_{F_2}$ which, by Remark
\ref{rem:homjoin}, is a homology ball of dimension equal to that of $F_1 \cup F_2$.
Moreover, $\sigma^{-1} (F) = \sigma_1^{-1} (F_1) \ast \sigma_2^{-1} (F_2)$ and hence
$\sigma^{-1} (F)$ is the interior of this ball.

Similarly, let $\sigma: \Delta' \to \Delta$ be a homology subdivision and let $F$
be a common face of $\Delta$ and $\Delta'$ (such as a vertex of $\Delta$) which
satisfies $\sigma(F) = F$. An easy application of part (ii) of Lemma
\ref{lem:homology} shows that
$\link_{\Delta'} (F)$ is a homology subdivision of $\link_\Delta (F)$ with
subdivision map $\sigma_F: \link_{\Delta'} (F) \to \link_\Delta (F)$ defined by
$\sigma_F (E) = \sigma(E \cup F) \sm F$. We will refer to this subdivision as
the link of $\sigma$ at $F$; its restriction to a face $G \in \link_\Delta (F)$
satisfies $(\link_{\Delta'} (F))_G = \link_{\Delta'_{F \cup G}} (F)$.

\smallskip
The following statement is an easy consequence of the relevant definitions; the
proof is left to the reader.

\begin{lemma} \label{lem:joinlink}
The simplicial join of two vertex-induced (respectively, flag) homology subdivisions
is also vertex-induced (respectively, flag). The link of a vertex-induced (respectively,
flag) homology subdivision is also vertex-induced (respectively, flag).
\end{lemma}

\medskip
\noindent
\textbf{Stellar subdivisions.}
We recall the following standard way to subdivide a simplicial complex. Given a
simplicial complex $\Delta$ on the ground set $\Omega$, a face $F \in \Delta$ and
an element $v$ not in $\Omega$, the \emph{stellar subdivision} of $\Delta$ on $F$
(with new vertex $v$) is the simplicial complex
  $$ \Delta' \ = \ ( \Delta \sm \st_\Delta (F) ) \ \cup \ ( \{v\} \ast
     \partial (2^F) \ast \link_\Delta (F) ) $$
on the ground set $\Omega \cup \{v\}$, where $\partial(2^F) = 2^F \sm \{F\}$. The
map $\sigma: \Delta' \to \Delta$, defined by
  $$ \sigma(E) \ = \ \begin{cases}
         E, & \text{if $E \in \Delta$ } \\
         (E \sm \{v\}) \cup F, & \text{otherwise} \end{cases} $$
for $E \in \Delta'$, is a topological (and thus a homology) subdivision of $\Delta$.
We leave to the reader to check that if $\Delta$ is a flag complex and $F \in
\Delta$ is an edge, then the stellar subdivision of $\Delta$ on $F$ is again a flag
complex.

  \section{Face enumeration, $\gamma$-vectors and local $h$-vectors}
  \label{sec:enu}

  This section reviews the definitions and main properties of the enumerative invariants
  of simplicial complexes and their subdivisions which will appear in the following
  sections, namely the $h$-vector of a simplicial complex, the $\gamma$-vector of an
  Eulerian simplicial complex and the local $h$-vector of a simplicial subdivision of
  a simplex. Some new results on local $h$-vectors are included.

  \subsection{$h$-vectors}
  \label{subsec:h}

  A fundamental enumerative invariant of a $(d-1)$-dimensional simplicial complex
  $\Delta$ is the $h$-polynomial, defined by
    $$ h(\Delta, x) \ = \ \sum_{F \in \Delta} \ x^{|F|} (1-x)^{d-|F|}. $$
  The $h$-vector of $\Delta$ is the sequence $h(\Delta) = (h_0 (\Delta),
  h_1 (\Delta),\dots,h_d (\Delta))$, where $h(\Delta, x) = \sum_{i=0}^d h_i (\Delta) x^i$.
  The number
    $$ (-1)^{d-1} h_d (\Delta) \ = \ \sum_{F \in \Delta} \, (-1)^{|F| - 1} $$
  is the reduced Euler characteristic of $\Delta$ and is denoted by $\widetilde{\chi}
  (\Delta)$. The polynomial $h(\Delta, x)$ satisfies $h_i (\Delta) = h_{d-i} (\Delta)$
  \cite[Section 3.14]{StaEC1} if $\Delta$ is an \emph{Eulerian} complex, meaning that
    $$ \widetilde{\chi} (\link_\Delta (F)) \ = \ (-1)^{\dim \, \link_\Delta (F)} $$
  holds for every $F \in \Delta$. For the simplicial join of two simplicial complexes
  $\Delta_1$ and $\Delta_2$ we have $h(\Delta_1 \ast \Delta_2, x) = h(\Delta_1, x)
  h(\Delta_2, x)$. For a homology ball or sphere $\Delta$ of dimension $d-1$ we set
    $$ h(\inte(\Delta), x) \ = \ \sum_{F \in \inte(\Delta)} \ x^{|F|}
       (1-x)^{d-|F|} $$
  and recall the following well-known statement (see, for instance, Theorem 7.1 in
  \cite[Chapter II]{StaCCA} and \cite[Section 2.1]{Ath10} for additional references).
    \begin{lemma} \label{lem:homformal}
      Let $\Delta$ be a $(d-1)$-dimensional simplicial complex. If $\Delta$ is
      either a homology ball or a homology sphere over $\kk$, then $x^d \, h (\Delta,
      1/x) = h(\inte(\Delta), x)$.
    \end{lemma}

  \subsection{$\gamma$-vectors} \label{subsec:gamma}
  Let $h = (h_0, h_1,\dots,h_d)$ be a vector with real coordinates and let $h(x) =
  \sum_{i=0}^d h_i x^i$ be the associated real polynomial of degree at most $d$. We say
  that \emph{$h(x)$ has symmetric coefficients}, and that the vector $h$ is
  \emph{symmetric}, if $h_i = h_{d-i}$ holds for $0 \le i \le d$. It is easy to check
  \cite[Proposition 2.1.1]{Ga05} that $h(x)$ has symmetric coefficients if and only if
  there exists a real polynomial $\gamma(x) = \sum_{i=0}^{\lfloor d/2 \rfloor} \gamma_i
  x^i$ of degree at most $\lfloor d/2 \rfloor$, satisfying
    \begin{equation} \label{eq:defgamma}
      h(x) \ = \ (1+x)^d \ \gamma \left( \frac{x}{(1+x)^2} \right) \, = \,
      \sum_{i=0}^{\lfloor d/2 \rfloor} \, \gamma_i x^i (1+x)^{d-2i}.
    \end{equation}
  In that case, $\gamma(x)$ is uniquely determined by $h(x)$ and called
  the \emph{$\gamma$-polynomial} associated to $h(x)$; the sequence $(\gamma_0,
  \gamma_1,\dots,\gamma_{\lfloor d/2 \rfloor})$ is called the \emph{$\gamma$-vector}
  associated to $h$. We will refer to the $\gamma$-polynomial associated to the
  $h$-polynomial of an Eulerian complex $\Delta$ as the \emph{$\gamma$-polynomial of
  $\Delta$} and will denote it by $\gamma (\Delta, x)$. Similarly, we will refer to
  the $\gamma$-vector associated to the $h$-vector of an Eulerian complex $\Delta$ as
  the \emph{$\gamma$-vector of $\Delta$} and will denote it by $\gamma (\Delta)$.

  \subsection{Local $h$-vectors}
  \label{sec:monolocal}

  We now recall some of the basics of the theory of face enumeration for subdivisions
  of simplicial complexes \cite{Sta92} \cite[Section III.10]{StaCCA}. The following
  definition is a restatement of \cite[Definition 2.1]{Sta92} for homology (rather
  than topological) subdivisions of the simplex.
    \begin{definition} \label{def:localh}
      Let $V$ be a set with $d$ elements and let $\Gamma$ be a homology subdivision
      of the simplex $2^V$. The polynomial $\ell_V (\Gamma, x) = \ell_0 + \ell_1 x +
      \cdots + \ell_d x^d$ defined by
        \begin{equation} \label{eq:deflocalh}
          \ell_V (\Gamma, x) \ = \sum_{F \subseteq V} \ (-1)^{d - |F|} \,
          h (\Gamma_F, x)
        \end{equation}
      is the \emph{local $h$-polynomial} of $\Gamma$ (with respect to $V$). The
      sequence $\ell_V (\Gamma) = (\ell_0, \ell_1,\dots,\ell_d)$ is the \emph{local
      $h$-vector} of $\Gamma$ (with respect to $V$).
    \end{definition}
  The following theorem, stated for homology subdivisions, summarizes some of the
  main properties of local $h$-vectors (see Theorems 3.2 and 3.3 and Corollary 4.7 in
  \cite{Sta92}).

    \begin{theorem} \label{thm:stalocal}
      \begin{itemize}
        \item[(i)]
          For every homology subdivision $\Delta'$ of a simplicial complex $\Delta$
          we have
            \begin{equation} \label{eq:hformula}
              h (\Delta', x) \ = \ \sum_{F \in \Delta} \,
              \ell_F (\Delta'_F, x) \, h (\link_\Delta (F), x).
            \end{equation}
        \item[(ii)]
          The local $h$-vector $\ell_V (\Gamma)$ is symmetric for every homology
          subdivision $\Gamma$ of the simplex $2^V$.
        \item[(iii)]
          The local $h$-vector $\ell_V (\Gamma)$ has nonnegative coordinates for every
          quasi-geometric homology subdivision $\Gamma$ of the simplex $2^V$.
      \end{itemize}
    \end{theorem}
    \begin{proof}
    Parts (i) and (iii) follow from the proofs of Theorems 3.2 and 4.6, respectively,
    in \cite{Sta92}. Moreover, Lemma \ref{lem:homformal} implies that every homology
    subdivision of a simplicial complex is a formal subdivision, in the sense of
    \cite[Definition 7.4]{Sta92}. Thus, parts (i) and (ii) are special cases of
    Corollary 7.7 and Theorem 7.8, respectively, in \cite{Sta92}.
    \end{proof}

    \begin{example} \label{ex:counter} The local $h$-polynomial of the subdivision in
    part (a) of Example \ref{ex:nonflagrest} was computed in \cite{Sta92} as $\ell_V
    (\Gamma, x) = -x^2$. This shows that the assumption in Theorem \ref{thm:stalocal}
    (iii) that $\Gamma$ is quasi-geometric is essential. For part (b) of Example
    \ref{ex:nonflagrest} we can easily compute that $\ell_V (\Gamma', x) = x + x^3$.
    Since $\Gamma'$ is quasi-geometric, this disproves \cite[Conjecture 5.4]{Sta92}
    (see also \cite[Section 6]{Cha94} \cite[p. 134]{StaCCA}), stating that local
    $h$-vectors of quasi-geometric subdivisions are unimodal. \qed
    \end{example}

  The previous example suggests the following question.

    \begin{question} \label{que:unimodal}
      Is the local $h$-vector $\ell_V (\Gamma)$ unimodal for every vertex-induced
      homology subdivision $\Gamma$ of the simplex $2^V$?
    \end{question}

  We now show that local $h$-vectors also enjoy a locality property (this will be
  useful in the proof of Proposition \ref{prop:edgesub}).

    \begin{proposition} \label{prop:localformula}
      Let $\sigma: \Gamma \to 2^V$ be a homology subdivision of the simplex $2^V$.
      For every homology subdivision $\Gamma'$ of $\Gamma$ we have
        \begin{equation} \label{eq:localformula}
          \ell_V (\Gamma', x) \ = \, \sum_{E \in \Gamma} \ \ell_E (\Gamma'_E, x) \,
          \ell_V (\Gamma, E, x),
        \end{equation}
      where
        \begin{equation} \label{eq:deflocalhrel}
          \ell_V (\Gamma, E, x) \ = \sum_{\sigma(E) \subseteq F \subseteq V} \
          (-1)^{d - |F|} \, h (\link_{\Gamma_F} (E), x)
        \end{equation}
      for $E \in \Gamma$.
    \end{proposition}
    \begin{proof}
      By assumption, $\Gamma'_F$ is a homology subdivision of $\Gamma_F$ for every
      $F \subseteq V$. Thus, using the defining equation (\ref{eq:deflocalh}) for
      $\ell_V (\Gamma', x)$ and (\ref{eq:hformula}) to expand $h (\Gamma'_F, x)$ for
      $F \subseteq V$, we get
        \begin{eqnarray*}
        \ell_V (\Gamma', x) &=& \sum_{F \subseteq V} \ (-1)^{d - |F|} \, h(\Gamma'_F,
        x) \\
        & & \\
        &=& \sum_{F \subseteq V} \ (-1)^{d - |F|} \
        \sum_{E \in \Gamma_F} \, \ell_E (\Gamma'_E, x) \, h (\link_{\Gamma_F} (E),
        x) \\
        & & \\
        &=& \sum_{E \in \Gamma} \, \ell_E (\Gamma'_E, x) \sum_{F \subseteq V: \,
        \sigma(E) \subseteq F} \ (-1)^{d - |F|} \, h (\link_{\Gamma_F} (E), x)
        \end{eqnarray*}
      and the proof follows.
    \end{proof}

    \begin{remark} \label{rem:relative}
      We call the polynomial $\ell_V (\Gamma, E, x)$, defined by (\ref{eq:deflocalhrel}),
      the \textit{relative local $h$-polynomial} of $\Gamma$ (with respect to $V$) at $E$.
      This polynomial reduces to $\ell_V (\Gamma, x)$ for $E = \varnothing$ and shares
      many of the important properties of $\ell_V (\Gamma, x)$, established in \cite{Sta92}.
      For instance, using ideas of \cite{Sta92} and their refinements in \cite{Ath10}, one
      can show that $\ell_V (\Gamma, E, x)$ has symmetric coefficients, in the sense that
        $$ x^{d-|E|} \, \ell_V (\Gamma, E, 1/x) \ = \ \ell_V (\Gamma, E, x), $$
      for every homology subdivision $\Gamma$ of $2^V$ and $E \in \Gamma$ and that $\ell_V
      (\Gamma, E, x)$ has nonnegative coefficients for every quasi-geometric homology
      subdivision $\Gamma$ of $2^V$ and $E \in \Gamma$. The following monotonicity
      property of local $h$-vectors is a consequence of the latter statement and
      (\ref{eq:localformula}): for every quasi-geometric homology subdivision $\Gamma$ of
      $2^V$ and every quasi-geometric homology subdivision $\Gamma'$ of $\Gamma$, we have
      $\ell_V (\Gamma', x) \ge \ell_V (\Gamma, x)$. Since these results will not be used in
      this paper, detailed proofs will appear elsewhere.
      \qed
    \end{remark}

  \section{Flag subdivisions of $\Sigma_{d-1}$}
  \label{sec:PL}

  This section proves Theorem \ref{thm:main}, as well as a result on flag subdivisions 
  of the simplex (Proposition \ref{prop:balltosphere}) which will be used in Section 
  \ref{sec:gamma}. 

  The following lemma gives several technical properties of homology balls and spheres 
  (over the field $\kk$). We will only sketch the proof, which is fairly 
  straightforward and uses standard tools from algebraic topology.

    \begin{lemma} \label{lem:homology}
      \begin{itemize}
        \item[(i)] If $\Delta$ is a homology sphere (respectively, ball) of dimension
        $d-1$, then $\link_\Delta (F)$ is a homology sphere of dimension $d-|F|-1$ for
        every $F \in \Delta$ (respectively, for every interior face $F \in \Delta$).
        \item[(ii)] If $\Delta$ is a homology ball of dimension $d-1$ and $F \in
        \Delta$ is a boundary face, then $\link_\Delta (F)$ is a homology ball of
        dimension $d-|F|-1$ with interior equal to $\link_\Delta (F) \cap \inte
        (\Delta)$.
        \item[(iii)] If $\Delta$ is a homology sphere (respectively, ball), then the
        cone over $\Delta$ is a homology ball whose boundary is equal to $\Delta$
        (respectively, to the union of $\Delta$ with the cone over the boundary of
        $\Delta$).
        \item[(iv)] Let $\Delta_1$ and $\Delta_2$ be homology balls of dimension
        $d$. If $\Delta_1 \cap \Delta_2$ is a homology ball of dimension $d-1$
        which is contained in (respectively, is equal to) the boundary of both
        $\Delta_1$ and $\Delta_2$,
        then $\Delta_1 \cup \Delta_2$ is a homology ball (respectively, sphere) of
        dimension $d$.
        \item[(v)] Let $\Delta$ be a homology sphere of dimension $d-1$. If
        $\Gamma$ is a subcomplex of $\Delta$ which is a homology ball of dimension
        $d-1$, then the complement of the interior of $\Gamma$ in $\Delta$ is also
        a homology ball of dimension $d-1$ whose boundary is equal to that of
        $\Gamma$.
      \end{itemize}
    \end{lemma}
    \begin{proof}
    We first observe that for all faces $F \in \Delta$ and $E \in \link_\Delta (F)$, 
    the link of $E$ in $\link_\Delta (F)$ is equal to $\link_\Delta (E \cup F)$. 
    Moreover, if $\Delta$ is a homology ball and $F$ is an interior face, then so 
    is $E \cup F$. Part (i) follows from these facts and the definition of homology
    balls and spheres. Part (ii) is an easy consequence of part (i) and the relevant 
    definitions. Part (iii) is an easy consequence of the relevant definitions and 
    the fact that cones have vanishing reduced homology. Part (iv) follows by an easy 
    application of the Mayer-Vietoris long exact sequence \cite[\S 25]{Mun84}.

    For the last part, we let $\RK$ denote the complement of the interior of $\Gamma$
    in $\Delta$ and note that the pairs $(\|\Gamma\|, \|\partial \Gamma\|)$ and
    $(\|\Delta\|, \|\RK\|)$ are compact triangulated relative homology manifolds which
    are orientable over $\kk$. Applying the Lefschetz duality theorem \cite[\S 70]{Mun84}
    and the long exact homology sequence \cite[\S 23]{Mun84} to these pairs shows that
    $\RK$ has trivial reduced homology over $\kk$. Similar arguments work for the
    links of faces of $\RK$. The details are omitted.
    \end{proof}

  \begin{remark} \label{rem:PL}
  Although not all parts of Lemma \ref{lem:homology} remain valid if homology balls
  and spheres are replaced by topological balls and spheres, they do hold for the
  subclasses of PL balls and PL spheres (we refer the reader to \cite[Section 4.7
  (d)]{BSZ99} for this claim, for a short introduction to PL topology and for
  additional references). Thus, the results of this paper remain valid when homology
  balls and spheres are replaced by PL balls and spheres and the notion of
  homology subdivision is replaced by its natural PL analogue. \qed
  \end{remark}

  The following lemma will also be essential in the proof of Theorem \ref{thm:main}. 
  A similar result has appeared in \cite[Lemma 3.2]{Bar10}.

\begin{lemma} \label{lem:starunion}
Let $\Delta$ be a flag $(d-1)$-dimensional homology sphere. For every nonempty
face $F$ of $\Delta$, the subcomplex $\bigcup_{v \in F} \overline{\st}_\Delta
(v)$ is a homology $(d-1)$-dimensional ball whose interior is equal to
$\bigcup_{v \in F} \st_\Delta (v)$.
\end{lemma}
\begin{proof}
We set $F = \{ v_1, v_2,\dots,v_k \}$, $\bigcup_{i=1}^k \overline{\st}_\Delta
(v_i) = \RK$ and $\bigcup_{i=1}^k \st_\Delta (v_i) = \RL$ and proceed by
induction on the cardinality $k$ of $F$. For $k=1$, the complex $\RK$ is the 
cone over $\link_\Delta (v_1)$ on the vertex $v_1$. Since $\Delta$ is a 
homology sphere, the result follows from parts (i) and (iii) of Lemma 
\ref{lem:homology}. Suppose that $k \ge 2$. We will also assume that $d \ge 3$, 
since the result is trivial otherwise (we note that the assumption that $\Delta$ 
is flag is essential in the case $d=2$). Since, by Lemma \ref{lem:homology} (i), 
links of flag homology spheres are also flag homology spheres, the 
complex $\Gamma = \link_\Delta (v_k)$ is a flag homology sphere of dimension 
$d-2$ and $\{ v_1,\dots,v_{k-1} \}$ is a nonempty face of $\Gamma$. Thus, 
by the induction hypothesis, the union $\Gamma_1 = \bigcup_{i=1}^{k-1} 
\overline{\st}_\Gamma (v_i)$ is a homology ball of dimension $d-2$. Let 
$\Gamma_0$ denote the boundary of $\Gamma_1$ and let $\Gamma_2$ denote the 
complement of the interior of $\Gamma_1$ in $\Gamma$. Thus $\Gamma_0$ is a 
homology sphere of dimension $d-3$ and, by part (v) of Lemma 
\ref{lem:homology}, $\Gamma_2$ is a homology ball of dimension $d-2$ whose 
boundary is equal to $\Gamma_0$.

      Consider the union $\RK_1 = \bigcup_{i=1}^{k-1} \overline{\st}_\Delta
      (v_i)$ and the cones $\RK_2 = v_k \ast \Gamma_2$ and $\RK_0 = v_k \ast
      \Gamma_0$. It is straightforward to verify that $\RK = \RK_1 \cup \RK_2$
      and that $\RK_1 \cap \RK_2 = \RK_0$. We note that $\RK_1$ is a homology
      ball of dimension $d-1$, by the induction hypothesis, and that $\RK_2$
      and $\RK_0$ are homology balls of dimension $d-1$ and $d-2$, respectively,
      by part (iii) of Lemma \ref{lem:homology}. By the induction
      hypothesis, the interior of $\Gamma_1$ is equal to $\bigcup_{i=1}^{k-1}
      \st_\Gamma (v_i)$. Therefore, none of the faces of $\Gamma_0$ contains
      any of $v_1,\dots,v_{k-1}$ and hence the same holds for $\RK_0$. Since,
      by the induction hypothesis, the interior of $\RK_1$ is equal to
      $\bigcup_{i=1}^{k-1} \st_\Delta (v_i)$, we conclude that $\RK_0$ is
      contained in the boundary of $\RK_1$. Moreover, $\RK_0$ is also contained
      in the boundary of $\RK_2$, since $\Gamma_0$ is contained in the boundary
      of $\Gamma_2$. It follows from the previous discussion and Lemma
      \ref{lem:homology} (iv) that $\RK$ is a homology $(d-1)$-dimensional
      ball.

      We now verify that the interior of $\RK$ is equal to $\RL$. This
      statement may be derived from the previous inductive argument, since the
      interior of $\RK$ is equal to the union of the interiors of $\RK_1$,
      $\RK_2$ and $\RK_0$. We give the following alternative argument. Since
      $\RK$ is a homology ball, its boundary consists of all faces of the
      $(d-2)$-dimensional faces of $\RK$ which are contained in exactly one
      facet of $\RK$. The validity of the statement for $k=1$ implies that
      these $(d-2)$-dimensional faces of $\RK$ are precisely those which do
      not contain any of the $v_i$ and which are not contained in more than
      one of the subcomplexes $\link_\Delta (v_i)$. However, since $\Delta$ is
      $(d-1)$-dimensional and flag, no $(d-2)$-dimensional face of $\Delta$
      may be contained in more than one of the $\link_\Delta (v_i)$. Thus, the
      boundary of $\RK$ consists precisely of its faces which do not contain
      any of the $v_i$ and the proof follows.
    \end{proof}

\medskip
\noindent
\begin{proof}[Proof of Theorem \ref{thm:main}] Let $\Delta$ be a flag simplicial
complex of dimension $d-1$ and $\Sigma_{d-1}$ be the simplicial join of the
zero-dimensional spheres $\{\varnothing, \{u_i\}, \{v_i\}\}$ for $1 \le i \le d$. 
We fix a facet $\{x_1, x_2,\dots,x_d\}$ of $\Delta$ and for $E \in \Delta$ 
we define
  \begin{equation} \label{eq:mainsigmadef}
    \sigma (E) \ = \ \{u_i: x_i \in E\} \cup \{v_i: E \notin
    \overline{\st}_\Delta (x_i) \}.
  \end{equation}
Clearly, $\sigma(E)$ cannot contain any of the sets $\{u_i, v_i\}$. Thus we have
$\sigma(E) \in \Sigma_{d-1}$ for every $E \in \Delta$ and hence we get a map $\sigma:
\Delta \to \Sigma_{d-1}$. We will prove that this map is a homology subdivision of
$\Sigma_{d-1}$, if $\Delta$ is a homology sphere. Given a face $F \in \Sigma_{d-1}$,
we need to show that $\sigma^{-1} (2^F)$ is a subcomplex of $\Delta$ of dimension
$\dim(F)$ which is a homology ball with interior $\sigma^{-1} (F)$. We denote by $S$
the subset of $\{x_1, x_2,\dots,x_d\}$ consisting of all vertices $x_i$ for which $F
\cap \{u_i, v_i\} = \varnothing$ and distinguish two cases:

\medskip
\noindent {\sf Case 1:} $S = \varnothing$. We may assume that $F =
\{u_1,\dots,u_k\} \cup \{v_{k+1},\dots,v_d\}$ for some $k \le d$.
Setting $E_0 = \{x_1,\dots,x_k\}$, the defining equation
(\ref{eq:mainsigmadef}) shows that $\sigma^{-1} (2^F)$ is equal to
the intersection of $\bigcap_{i=1}^k \overline{\st}_\Delta (x_i)$
with the complement of $\bigcup_{i=k+1}^d \st_\Delta (x_i)$ in
$\Delta$ and that $\sigma^{-1} (F)$ consists of those faces of
$\sigma^{-1} (2^F)$ which contain $E_0$ and do not belong to any
of the $\link_\Delta (x_i)$ for $k+1 \le i \le d$. Consider the
complex $\Gamma = \link_\Delta (E_0)$ and let $\RK$ denote the
complement of the union $\bigcup_{i=k+1}^d \st_\Gamma (x_i)$ in
$\Gamma$. Since links of homology spheres are also homology spheres 
(see part (i) of Lemma \ref{lem:homology}), the complex $\Gamma$ 
is a homology sphere of dimension $d-|F|-1$. By Lemma
\ref{lem:starunion} and part (v) of Lemma \ref{lem:homology},
$\RK$ is a homology ball of dimension $d-|F|-1$ whose interior is
equal to the set of those faces of $\RK$ which do not belong to
any of the $\link_\Gamma (x_i)$ for $k+1 \le i \le d$. From the
above we conclude that $\sigma^{-1} (2^F)$ is equal to the
simplicial join of the simplex $2^{E_0}$ and $\RK$ and that
$\sigma^{-1} (F)$ is equal to the simplicial join of $\{E_0\}$ and
the interior of $\RK$. The result now follows from part (iii) of
Lemma \ref{lem:homology} and the previous discussion.

\medskip
\noindent {\sf Case 2:} $S \ne \varnothing$. Then $\sigma^{-1} (2^F)$ is contained
in $\link_\Delta (S)$. As a result, replacing $\Delta$ by $\link_\Delta (S)$ reduces
this case to the previous one.

\medskip
Finally, we note that (\ref{eq:mainsigmadef}) may be rewritten as $\sigma (E) =
\bigcup_{x \in E} f(x)$, where
  $$ f(x) \, = \, \begin{cases} \{u_i\}, & \text{if \ $x=x_i$ for some $1 \le i \le d$}
     \\ \{ v_i: x \notin \link_\Delta (x_i) \}, & \text{otherwise} \end{cases} $$
for every vertex $x$ of $\Delta$. This implies that for every $E \in \Delta$, the
carrier of $E$ is equal to the union of the carriers of the vertices of $E$. As a
result, $\sigma$ is vertex-induced and the proof follows.
\end{proof}

    \begin{corollary} \label{cor:main1}
      Given any flag homology sphere $\Delta$ of dimension $d-1$, there exist
      simplicial complexes $\Gamma_F$, one for each face $F \in \Sigma_{d-1}$,
      with the following properties: {\rm (a)} $\Gamma_F$ is a flag vertex-induced
      homology subdivision of the simplex $2^F$ for every $F \in \Sigma_{d-1}$;
      and {\rm (b)} we have
        \begin{equation} \label{eq:corPL}
          h (\Delta, x) \ = \ \sum_{F \in \Sigma_{d-1}} \,
          \ell_F (\Gamma_F, x) \, (1+x)^{d - |F|}.
        \end{equation}
    \end{corollary}

    \begin{proof}
      This statement follows by applying (\ref{eq:hformula}) to the subdivision of
      $\Delta$ guaranteed by Theorem \ref{thm:main} and noting that for every $F \in
      \Sigma_{d-1}$, the restriction $\Gamma_F$ of this subdivision to $F$ has the
      required properties and that $h (\link_{\Sigma_{d-1}} (F), x) = (1+x)^{d - |F|}$.
    \end{proof}

    \begin{remark} \label{rem:hvecmin}
      It follows from (\ref{eq:corPL}) and Theorem \ref{thm:stalocal} (iii) that
      $h(\Delta, x) \ge (1+x)^d$ for every flag $(d-1)$-dimensional homology sphere
      $\Delta$. This inequality was proved, more generally, for every flag
      $(d-1)$-dimensional doubly Cohen-Macaulay simplicial complex $\Delta$ in
      \cite[Theorem 1.3]{Ath11}.
      \qed
    \end{remark}

We now fix a $d$-element set $V = \{v_1, v_2,\dots,v_d\}$ and a homology
subdivision $\Gamma$ of $2^V$, with subdivision map $\sigma: \Gamma \to 2^V$.
We let $U = \{u_1, u_2,\dots,u_d\}$ be a $d$-element set which is disjoint from
$V$ and consider the union $\Delta$ of all collections of the form
$2^E \ast \Gamma_G$, where $E = \{ u_i: i \in I\}$ and $G = \{ v_j: j \in J\}$ are
subsets of $U$ and $V$, respectively, and $(I, J)$ ranges over all ordered pairs
of disjoint subsets of $\{1, 2,\dots,d\}$. Clearly, $\Delta$ is a simplicial complex
which contains as a subcomplex $\Gamma$ (set $I = \varnothing$) and the simplex
$2^U$ (set $J = \varnothing$).

We let $\Sigma_{d-1}$ be as in the proof of Theorem \ref{thm:main} and define the
map $\sigma_0: \Delta \to \Sigma_{d-1}$ by $\sigma_0 (E \cup F) = E \cup \sigma(F)$
for all $E \subseteq U$ and $F \in \Gamma$ such that $E \cup F \in \Delta$. The
second result of this section is as follows.

\begin{proposition} \label{prop:balltosphere}
Under the established assumptions and notation, the following hold:
  \begin{itemize}
  \item[(i)]
  The complex $\Delta$ is a $(d-1)$-dimensional homology sphere.
  \item[(ii)]
  Endowed with the map $\sigma_0$, the complex $\Delta$ is a homology subdivision
  of $\Sigma_{d-1}$.
  \item[(iii)]
  If $\Gamma$ is flag and vertex-induced, then $\Delta$ is a flag simplicial
  complex and a flag, vertex-induced homology subdivision of $\Sigma_{d-1}$.
  \end{itemize}
\end{proposition}
\begin{proof}
We first verify (ii). We consider any face $W \in \Sigma_{d-1}$,
so that $W = E \cup G$ for some $E \subseteq U$ and $G \subseteq
V$, and recall that $\Gamma_G$ is a homology ball of dimension
$\dim(G)$. By definition of $\sigma_0$ we have $\sigma_0^{-1}
(2^W) = 2^E \ast \Gamma_G$ and $\sigma_0^{-1} (W) = \{E\} \ast
\sigma^{-1} (G) = \{E\} \ast \inte (\Gamma_G)$. Thus, it follows
from part (iii) of Lemma \ref{lem:homology} that $\sigma_0^{-1}
(2^W)$ is a homology ball of dimension $\dim(W)$ and that its
interior is equal to $\sigma_0^{-1} (W)$.

Part (i) may be deduced from part (ii) as follows. Let $F_0,
F_1,\dots,F_m$ be a linear ordering of the facets of
$\Sigma_{d-1}$ such that $F_i \cap U \subset F_j \cap U$ implies
$i < j$. Thus we have $m = 2^d$, $F_0 = V$ and $F_m = U$. By
assumption, $\Delta_{F_0} = \Gamma_V$ is a $(d-1)$-dimensional
homology ball. Moreover, $\Delta_{F_j}$ is equal to the simplicial
join of a face of $2^U$ with the restriction of $\Gamma$ to a face
of $2^V$, for $1 \le j \le m$, and hence a $(d-1)$-dimensional
homology ball by part (iii) of Lemma \ref{lem:homology}, and
$\Delta_{F_j} \cap \bigcup_{i=0}^{j-1} \Delta_{F_i}$ is equal to
the simplicial join of the boundary of this face with the same
restriction of $\Gamma$. It follows from part (iv) of Lemma
\ref{lem:homology} by induction on $j$ that $\bigcup_{i=0}^j
\Delta_{F_i}$ is a $(d-1)$-dimensional homology ball, for $0 \le j
\le m-1$ and a $(d-1)$-dimensional homology sphere, for $j = m$.
This proves (i), since $\Delta = \bigcup_{i=0}^m \Delta_{F_i}$.

To verify (iii), assume that $\Gamma$ is flag and vertex-induced. It is clear from
the definition of $\sigma_0$ that the subdivision $\Delta$ is also vertex-induced.
Since the restriction of $\Delta$ to any face of $\Sigma_{d-1}$ is the join of a
simplex with the restriction of $\Gamma$ to a face of $2^V$, the
subdivision $\Delta$ is flag as well. To verify that $\Delta$ is a flag complex,
let $E \cup F$ be a set of vertices of $\Delta$ which are pairwise joined
by edges, where $E = \{ u_i: i \in I \}$ for some $I \subseteq \{1, 2,\dots,d\}$
and $F$ consists of vertices of $\Gamma$. We need to show that $E \cup F \in
\Delta$. We set $J = \{1, 2,\dots,d\} \sm I$ and $G = \{ v_j: j \in J\}$ and note
that the elements of $F$ are vertices of $\Gamma_G$, by definition of $\Delta$.
Since the elements of $F$ are pairwise joined by edges in $\Gamma$, our assumptions
that $\Gamma$ is vertex-induced and flag imply that $F \in \Gamma_G$. Therefore
$E \cup F$ belongs to $2^E \ast \Gamma_G$, which is contained in $\Delta$, and
the result follows.
\end{proof}

\begin{remark} \label{rem:vineeded}
The conclusion in Proposition \ref{prop:balltosphere} that $\Delta$ is a flag
complex does not hold under the weaker hypothesis that $\Gamma$ is quasi-geometric,
rather than vertex-induced. For instance, let $\Gamma$ be the simplicial complex
consisting of the subsets of $V = \{v_1, v_2, v_3\}$ and $\{v_2, v_3, v_4\}$ and
let $\sigma: \Gamma \to 2^V$ be the subdivision which pushes $\Gamma$ into $2^V$,
so that the face $F = \{v_2, v_3\}$ of $\Gamma$ ends up in the interior of $2^V$ and
$v_4$ ends up in the interior of $2^F$. Then $\Gamma$ is quasi-geometric and flag
but the simplicial complex $\Delta$ is not flag, since it has $\{u_1, v_2, v_3\}$
as a minimal non-face.
\qed
\end{remark}

  \section{Local $\gamma$-vectors}
  \label{sec:gamma}

  This section defines the local
  $\gamma$-vector of a homology subdivision of the simplex, lists
  examples and elementary properties, discusses its nonnegativity in the special
  case of flag subdivisions and concludes with the proof of Theorem \ref{thm:34}.
  This proof comes as an application of the considerations and results of the
  present and the previous section.

  \begin{definition} \label{def:localgamma}
      Let $V$ be a set with $d$ elements and let $\Gamma$ be a homology subdivision
      of the simplex $2^V$. The polynomial $\xi_V (\Gamma, x) = \xi_0 + \xi_1 x +
      \cdots + \xi_{\lfloor d/2 \rfloor} x^{\lfloor d/2 \rfloor}$ defined by
        \begin{equation} \label{eq:defxi}
          \ell_V (\Gamma, x) \ = \ (1+x)^d \ \xi_V \left( \Gamma, \frac{x}{(1+x)^2}
          \right) \, = \, \sum_{i=0}^{\lfloor d/2 \rfloor} \, \xi_i x^i (1+x)^{d-2i}
        \end{equation}
      is the \emph{local $\gamma$-polynomial} of $\Gamma$ (with respect to $V$).
      The sequence $\xi_V (\Gamma) = (\xi_0, \xi_1,\dots,\xi_{\lfloor d/2 \rfloor})$
      is the \emph{local $\gamma$-vector} of $\Gamma$ (with respect to $V$).
    \end{definition}

  Thus $\xi_V (\Gamma, x)$ is the $\gamma$-polynomial associated to $\ell_V (\Gamma,
  x)$ and $\xi_V (\Gamma)$ is the $\gamma$-vector associated to $\ell_V (\Gamma)$,
  in the sense of Section \ref{subsec:gamma}.
  All formulas in the next example follow from corresponding formulas in \cite[Example
  2.3]{Sta92}, or directly from the relevant definitions.

  \begin{example} \label{ex:localgamma}
    (a) For the trivial subdivision $\Gamma = 2^V$ of the $(d-1)$-dimensional simplex
    $2^V$ we have
      \begin{equation} \label{eq:extrivial}
        \xi_V (\Gamma, x) \ = \ \begin{cases}
        1, & \text{if \ $d=0$} \\
        0, & \text{if \ $d \ge 1$}. \end{cases}
      \end{equation}

    (b) Let $\xi_V (\Gamma) = (\xi_0, \xi_1,\dots,\xi_{\lfloor d/2 \rfloor})$, where
    $\Gamma$ and $V$ are as in Definition \ref{def:localgamma}. Assuming that $d \ge 1$,
    we have $\xi_0 = 0$ and $\xi_1 = f_0^\circ$, where $f_0^\circ$ is the number of
    interior vertices of $\Gamma$. Assuming that $d \ge 4$, we also have $\xi_2 =
    -(2d-3) f_0^\circ + f_1^\circ - \tilde{f}_0$, where $f_1^\circ$ is the number of
    interior edges of $\Gamma$ and $\tilde{f}_0$ is the number of vertices of $\Gamma$
    which lie in the relative interior of a $(d-2)$-dimensional face of $2^V$.

    (c) Suppose that $d \in \{2, 3\}$. As a consequence of (b) we have $\xi_V (\Gamma,
    x) = tx$ for every homology subdivision $\Gamma$ of $2^V$, where $t$ is the number
    of interior vertices of $\Gamma$.

    (d) Let $\Gamma$ be the cone over the boundary $2^V \sm \{V\}$ of the simplex $2^V$
    (so $\Gamma$ is the stellar subdivision of $2^V$ on the face $V$). Then
    $\ell_V (\Gamma, x) = x + x^2 + \cdots + x^{d-1}$ and hence $\xi_2$ is negative for
    $d \ge 4$. For instance, we have $\xi_V (\Gamma, x) = x - x^2$ for $d=4$.

    (e) For the subdivisions of parts (b) and (c) of Example \ref{ex:nonflagrest} we
    can compute that $\ell_V (\Gamma', x) = \ell_V (\Gamma'', x) = x + x^3$ and hence
    that $\xi_V (\Gamma', x) = \xi_V (\Gamma'', x) = x - 2x^2$. \qed
  \end{example}

 The following proposition shows the relevance of local $\gamma$-vectors in the study
 of $\gamma$-vectors of subdivisions of Eulerian complexes.

\begin{proposition} \label{prop:gammaformula}
Let $\Delta$ be an Eulerian simplicial complex. For every homology subdivision $\Delta'$
of $\Delta$ we have
  \begin{equation} \label{eq:gammaformula}
    \gamma (\Delta', x) \ = \ \sum_{F \in \Delta} \,
    \xi_F (\Delta'_F, x) \, \gamma (\link_\Delta (F), x).
  \end{equation}
\end{proposition}
\begin{proof}
Since $\Delta$ is Eulerian, so is $\link_\Delta (F)$ for every $F \in \Delta$.
Thus, applying (\ref{eq:defgamma}) to the $h$-polynomial of $\link_\Delta (F)$
we get
  $$ h (\link_\Delta (F), x) \ = \ (1+x)^{d-|F|} \ \gamma \left( \link_\Delta (F),
     \frac{x}{(1+x)^2} \right), $$
where $d-1 = \dim(\Delta)$. Using this and (\ref{eq:defxi}), Equation (\ref{eq:hformula})
may be rewritten as
  $$ h (\Delta', x) \ = \ (1+x)^d \ \sum_{F \in \Delta} \, \xi_F \left( \Delta'_F,
     \frac{x}{(1+x)^2} \right) \, \gamma \left( \link_\Delta (F), \frac{x}{(1+x)^2}
     \right). $$
The proposed equality now follows from the uniqueness of the $\gamma$-polynomial
associated to $h (\Delta', x)$.
\end{proof}

The following statement is the main conjecture of this paper.

\begin{conjecture} \label{conj:main}
For every flag vertex-induced homology subdivision $\Gamma$ of the simplex $2^V$
we have $\xi_V (\Gamma) \ge 0$.
\end{conjecture}

Parts (d) and (e) of Example \ref{ex:localgamma} show that the conclusion of Conjecture
\ref{conj:main} fails under various weakenings of the hypotheses. We do not know of an
example of a flag quasi-geometric homology subdivision of the simplex for which the local
$\gamma$-vector fails to be nonnegative.

We now discuss some consequences of Theorem \ref{thm:main} and Proposition
\ref{prop:gammaformula}, related to Conjecture \ref{conj:main}.

\begin{corollary} \label{cor:main12}
For every flag homology sphere $\Delta$ of dimension $d-1$ we have
  \begin{equation} \label{eq:corPL2}
    \gamma (\Delta, x) \ = \ \sum_{F \in \Sigma_{d-1}} \, \xi_F (\Gamma_F, x),
  \end{equation}
where $\Gamma_F$ is as in Corollary \ref{cor:main1} for each $F \in \Sigma_{d-1}$.
In particular, the validity of Conjecture \ref{conj:main} for homology subdivisions
$\Gamma$ of dimension at most $d-1$ implies the validity of Conjecture
\ref{conj:gal} for homology spheres $\Delta$ of dimension at most $d-1$.
\end{corollary}

\begin{proof}
Setting $\ell_F (\Gamma_F, x) = \sum_i \xi_{F, i} \, x^i (1+x)^{|F| - 2i}$
in (\ref{eq:corPL}) and changing the order of summation, results in
(\ref{eq:corPL2}). Alternatively, one can apply (\ref{eq:gammaformula}) to
the subdivision guaranteed by Theorem \ref{thm:main} and note that $\gamma
(\link_{\Sigma_{d-1}} (F), x) = 1$ for every $F \in \Sigma_{d-1}$. The last
sentence in the statement of the corollary follows from (\ref{eq:corPL2}).
\end{proof}

\begin{corollary} \label{cor:main2}
The validity of Conjecture \ref{conj:main} for homology subdivisions $\Gamma$
of dimension at most $d-1$ implies the validity of Conjecture
\ref{conj:newmonotone} for homology spheres $\Delta$ and subdivisions $\Delta'$
of dimension at most $d-1$.
\end{corollary}
\begin{proof}
We observe that the term corresponding to $F = \varnothing$ in the sum of the
right-hand side of (\ref{eq:gammaformula}) is equal to $\gamma(\Delta, x)$. \
Thus, the result follows from (\ref{eq:gammaformula}), Corollary \ref{cor:main12}
and the fact that the link of every nonempty face of a flag homology sphere is
also a flag homology sphere of smaller dimension.
\end{proof}

\begin{proposition} \label{prop:main3d}
Conjecture \ref{conj:main} holds for subdivisions of the 3-dimensional simplex.
\end{proposition}
\begin{proof}
Let $\Gamma$ be a flag vertex-induced homology subdivision of the
$(d-1)$-dimensional simplex $2^V$ and let $\Delta$ be the homology subdivision
of $\Sigma_{d-1}$ considered in Proposition \ref{prop:balltosphere}. Applying
(\ref{eq:gammaformula}) to this subdivision and noting that $\gamma
(\link_{\Sigma_{d-1}} (F), x) = 1$ for every $F \in \Sigma_{d-1}$, we get
\[ \gamma (\Delta, x) \ = \ \sum_{F \in \Sigma_{d-1}} \, \xi_F (\Delta_F, x).
\]
By definition of $\Delta$, the restriction $\Delta_F$ is a cone over the
restriction of $\Delta$ to a proper face of $F$ for every $F \in \Sigma_{d-1}$
which is not contained in $V$. Since every such subdivision has a zero local
$h$-vector \cite[p.~821]{Sta92}, the previous formula can be rewritten as
  \begin{equation} \label{eq:corPL3}
    \gamma (\Delta, x) \ = \ \sum_{F \subseteq V} \, \xi_F (\Gamma_F, x).
  \end{equation}

Assume now that $d=4$, so that $\xi (\Gamma, x) = \xi_0 + \xi_1 x + \xi_2 x^2$ for
some integers $\xi_0, \xi_1, \xi_2$. Since $\xi_0 = 0$ and $\xi_1 \ge 0$ by part
(b) of Example \ref{ex:localgamma}, it suffices to show that $\xi_2 \ge 0$. For
that, we observe that the only contribution to the coefficient of $x^2$ in the
right-hand side of (\ref{eq:corPL3}) comes from the term with $F = V$. As a result,
$\xi_2$ is equal to the coefficient of $x^2$ in $\gamma (\Delta, x)$. Since,
$\Delta$ is a 3-dimensional flag homology sphere (by Proposition
\ref{prop:balltosphere}), this coefficient is nonnegative by the Davis-Okun
theorem \cite[Theorem 11.2.1]{DOk01} and the result follows.
\end{proof}

\medskip
\noindent
\emph{Proof of Theorem \ref{thm:34}}. For 3-dimensional spheres the result follows
from Proposition \ref{prop:main3d} and Corollary \ref{cor:main2}.
Assume now that $\Delta$ and $\Delta'$ have dimension 4. Then we can write $\gamma
(\Delta, x) = 1 + \gamma_1 (\Delta) x + \gamma_2 (\Delta) x^2$ and $\gamma (\Delta',
x) = 1 + \gamma_1 (\Delta') x + \gamma_2 (\Delta') x^2$. Since $\gamma_1 (\Delta) =
f_0 (\Delta) - 8$ and $\gamma_1 (\Delta') = f_0 (\Delta') - 8$, where $f_0 (\Delta)$
and $f_0 (\Delta')$ is the number of vertices of $\Delta$ and $\Delta'$, respectively,
it is clear that $\gamma_1 (\Delta') \ge \gamma_1 (\Delta)$. As the computation in
the proof of \cite[Corollary 2.2.2]{Ga05} shows, we also have
\[ 2 \gamma_2 (\Delta) \ = \ \sum_{v \in \vset(\Delta)} \, \gamma_2 (\link_\Delta
(v)), \]
where $\vset(\Delta)$ is the set of vertices of $\Delta$. Similarly, we have
\[ 2 \gamma_2 (\Delta') \ = \ \sum_{v' \in \vset(\Delta')} \, \gamma_2
(\link_{\Delta'} (v')), \]
where we may assume that $\vset(\Delta) \subseteq \vset(\Delta')$. Since
$\link_{\Delta'} (v)$ is a
flag vertex-induced homology subdivision of $\link_{\Delta} (v)$ for every $v \in
\vset(\Delta)$, by Lemma \ref{lem:joinlink}, we have $\gamma_2 (\link_{\Delta'} (v))
\ge \gamma_2 (\link_\Delta (v))$ by the 3-dimensional case, treated earlier, for
every such vertex $v$. Since $\link_{\Delta'} (v')$ is a 3-dimensional flag homology
sphere, we also have $\gamma_2 (\link_{\Delta'} (v')) \ge 0$ by the Davis-Okun
theorem for every $v' \in \vset(\Delta') \sm \vset(\Delta)$. Hence $\gamma_2
(\Delta') \ge \gamma_2 (\Delta)$ and the result follows.
\qed

    \begin{question} \label{que:monotone}
      Does $\gamma (\Delta') \ge \gamma (\Delta)$ hold for every flag homology sphere
      $\Delta$ and every flag homology subdivision $\Delta'$ of $\Delta$?
    \end{question}

  \section{Special cases}
  \label{sec:evidence}

  This section provides some evidence in favor of the validity of Conjecture
  \ref{conj:main} other than that provided by Proposition \ref{prop:main3d}.

  \medskip
  \noindent
  \textbf{Simplicial joins.} Let $\Gamma$ be a homology subdivision of the simplex $2^V$
  and $\Gamma'$ be a homology subdivision of the simplex $2^{V'}$, where $V$ and $V'$ are
  disjoint finite sets. Then $\Gamma \ast \Gamma'$ is a homology subdivision of
  the simplex $2^V \ast 2^{V'} = 2^{V \cup V'}$ and given subsets $F \subseteq V$ and $F'
  \subseteq V'$, the restriction of $\Gamma \ast \Gamma'$ to the face $F \cup F'$ of this
  simplex satisfies $(\Gamma \ast \Gamma')_{F \cup F'} = \Gamma_F \ast \Gamma'_{F'}$.
  Since $h ( \Gamma_F \ast \Gamma'_{F'}, x) = h (\Gamma_F, x) \, h (\Gamma'_{F'}, x)$,
  the defining equation (\ref{eq:deflocalh}) and a straightforward computation show that
    $$ \ell_{V \cup V'} \, (\Gamma \ast \Gamma', x)   \ = \ \ell_V (\Gamma, x) \, \ell_{V'}
       (\Gamma', x). $$
  This equation and (\ref{eq:defxi}) imply that
    \begin{equation} \label{eq:xiproduct}
      \xi_{V \cup V'} \, (\Gamma \ast \Gamma', x) \ = \ \xi_V (\Gamma, x) \, \xi_{V'}
      (\Gamma', x).
    \end{equation}
  From the previous formula and Lemma \ref{lem:joinlink} we conclude that if $\Gamma$ and
  $\Gamma'$ satisfy the assumptions and the conclusion of Conjecture \ref{conj:main}, then
  so does $\Gamma \ast \Gamma'$.

  \medskip
  \noindent
  \textbf{Edge subdivisions.} Following \cite[Section 5.3]{CD95}, we refer to the
  stellar subdivision on an edge of a simplicial complex $\Gamma$ as an \emph{edge
  subdivision}. As mentioned in Section \ref{subsec:gamma}, flagness of a simplicial
  complex is preserved by edge subdivisions. The following statement describes a class
  of flag (geometric) subdivisions of the simplex with nonnegative local
  $\gamma$-vectors.
    \begin{proposition} \label{prop:edgesub}
      For every subdivision $\Gamma$ of the simplex $2^V$ which can be obtained from
      the trivial subdivision by successive edge subdivisions, we have $\xi_V (\Gamma)
      \ge 0$.
    \end{proposition}

    \begin{proof}
      Let $\Gamma$ be a subdivision of $2^V$ and $\Gamma'$ be the edge subdivision of
      $\Gamma$ on $e = \{a, b\} \in \Gamma$. Thus we have $\Gamma' = ( \Gamma \sm
      \st_\Gamma(e) ) \cup ( \{v\} \ast \partial(e) \ast \link_\Gamma (e) )$, where
      $v$ is the new vertex added and $\partial(e) = \{ \varnothing, \{a\}, \{b\} \}$.

      By appealing to (\ref{eq:localformula}) and noticing that the right-hand side of
      this formula vanishes except when $E \in \{ \varnothing, e\}$ (or by direct
      computation), we find that
        $$ \ell_V (\Gamma', x) \ = \ \ell_V (\Gamma, x) \, + \, x \, \ell_V (\Gamma,
           e, x). $$
      Thus it suffices to prove the following claim: if the $\gamma$-polynomial
      corresponding to $\ell_V (\Gamma, E, x)$ has nonnegative coefficients for every
      face $E \in \Gamma$ of positive dimension (meaning that $\ell_V (\Gamma, E, x)$
      can be written as a linear combination of the polynomials $x^i (1+x)^{d - |E| -
      2i}$ with nonnegative coefficients for every $|E| \ge 2$), then the same holds
      for $\Gamma'$. We consider a face $E \in \Gamma'$ of positive dimension and
      distinguish the following cases (we note that $E$ cannot contain $e$ and that if
      $E \in \Gamma$, then the carrier $\sigma(E) \subseteq V$ of $E$ is the same,
      whether considered with respect to $\Gamma$ or $\Gamma'$):

      \medskip
      \noindent {\sf Case 1:} $E \in \Gamma \sm \link_\Gamma (e)$. The links
      $\link_{\Gamma'_F} (E)$ and $\link_{\Gamma_F} (E)$ are then combinatorially
      isomorphic for every $F \subseteq V$ which contains the carrier of $E$ (these
      two links are equal if $E \cup e \notin \Gamma$) and the defining equation
      (\ref{eq:deflocalhrel}) implies that $\ell_V (\Gamma', E, x) = \ell_V (\Gamma,
      E, x)$.

      \medskip
      \noindent {\sf Case 2:} $E \in \link_\Gamma (e)$. For $F \subseteq V$ which
      contains the carrier of $E$, the link $\link_{\Gamma'_F} (E)$ is equal to either
      $\link_{\Gamma_F} (E)$ or to the edge subdivision of $\link_{\Gamma_F} (E)$ on
      $e$, in case $F$ does not or does contain the carrier of $e$, respectively. It
      follows from this and (\ref{eq:hformula}) that (see also
      \cite[Proposition 2.4.3]{Ga05})
        $$ h(\link_{\Gamma'_F} (E), x) \ = \ \begin{cases}
        h(\link_{\Gamma_F} (E), x), & \text{if \ $\sigma(e) \not\subseteq F$} \\
        h(\link_{\Gamma_F} (E), x) \, + \, x \, h(\link_{\Gamma_F} (E \cup e), x), &
        \text{if \ $\sigma(e) \subseteq F$} \end{cases} $$
      and then from (\ref{eq:deflocalhrel}) that $\ell_V (\Gamma', E, x) = \ell_V
      (\Gamma, E, x) + x \ell_V (\Gamma, E \cup e, x)$.

      \medskip
      \noindent {\sf Case 3:} $E \notin \Gamma$. Then we must have $E \in \{v\} \ast
      \partial(e) \ast \link_\Gamma (e)$ and, in particular, $v \in E$. We distinguish
      two subcases:

      Suppose first that $E$ intersects $e$ and set $E' = (E \sm \{v\}) \cup e$. Then
      $\link_{\Gamma'_F} (E) = \link_{\Gamma_F} (E')$ for every $F \subseteq V$ which
      contains the carrier of $E$ in $\Gamma'$ (the latter coincides with the carrier
      of $E'$ in $\Gamma$) and hence $\ell_V (\Gamma', E, x) = \ell_V (\Gamma, E', x)$.

      Suppose finally that $E \cap e = \varnothing$ and set $E' = (E \sm \{v\}) \cup e$.
      Then $\link_{\Gamma'_F} (E) = \link_{\Gamma_F} (E') \ast \partial(e)$ for every
      $F \subseteq V$ which contains the carrier of $E$ in $\Gamma'$. Therefore we have
      $h(\link_{\Gamma'_F} (E) , x) = (1+x) \, h(\link_{\Gamma_F} (E') , x)$ for every
      such $F$ and hence $\ell_V (\Gamma', E, x) = (1+x) \, \ell_V (\Gamma, E', x)$.

      \medskip
      The expressions obtained for $\ell_V (\Gamma', E, x)$ and our assumption on
      $\Gamma$ show that indeed, the corresponding $\gamma$-polynomial has nonnegative
      coefficients in all cases.
    \end{proof}

\medskip
\noindent
\textbf{Barycentric and cluster subdivisions.}  As a special case of Proposition
\ref{prop:edgesub}, the (first) barycentric subdivision of the simplex $2^V$ has
nonnegative local $\gamma$-vector. Several combinatorial interpretations for its
entries are given in \cite{AS11}. Similar results appear there for the simplicial
subdivision of a simplex defined by the positive part of the cluster complex,
associated to a finite root system.

The following special case of Conjecture \ref{conj:main} might also be
interesting to explore. The notion of a CW-regular subdivision can be defined by
replacing the simplicial complex $\Delta'$ in the definition of a topological
subdivision (Definition \ref{def:sub}) by a regular CW-complex; see \cite[p. 839]{Sta92}.
  \begin{question} \label{que:bary}
    Does Conjecture \ref{conj:main} hold for the barycentric subdivision of any
    CW-regular subdivision of the simplex?
  \end{question}

  \section*{Acknowledgments} The author wishes to thank an anonymous referee for
  helpful suggestions and Mike Davis and Volkmar Welker for useful discussions.

  \end{document}